\documentclass{amsart}[12pt]
\usepackage{graphicx,graphics, amsmath, amsthm, amscd, amsfonts}

\usepackage{latexsym}
\makeatletter \oddsidemargin.9375in \evensidemargin \oddsidemargin
\newtheorem{theorem}{Theorem}[section]
\newtheorem{lemma}[theorem]{Lemma}

\theoremstyle{definition}

\newtheorem{remark}[theorem]{Remark}
\numberwithin{equation}{section}

\begin{document}
\Large
\title[Maps completely preserving involution]
{Maps completely preserving involution}

\author[Ali Taghavi and Roja Hosseinzadeh]{Ali Taghavi and Roja
Hosseinzadeh}

\address{{ Department of Mathematics, Faculty of Basic Sciences,
 University of Mazandaran, P. O. Box 47416-1468, Babolsar, Iran.}}

\email{taghavi@umz.ac.ir,  ro.hosseinzadeh@umz.ac.ir}



\begin{abstract}\large
Let $X$ and $Y$ be Banach spaces
and $\mathcal{A} \subseteq B(X)$, $\mathcal{B} \subseteq B(Y)$ be
standard operator algebras with $ \dim X \geq 3$. We characterize the form of maps from $\mathcal{A}$ onto $\mathcal{B}$ such that
completely preserve involution.
\end{abstract} \maketitle

\section{Introduction And Statement of the Results}
\noindent The study of maps on operator algebras
preserving certain properties or subsets is a topic which
attracts much attention of many authors.  See for example the introduction to this topic in the paper $[12]$ and the references within and also $[1,8,11,13,14]$. Recently, some of these problems
are concerned with completely preserving of certain properties or subsets of operators. For example see $[3-7,9,10]$.
In $[6]$, authors characterized completely rank-nonincreasing linear maps and then later extended this results in $[7]$.
Completely invertibility preserving maps were characterized in $[3,4,9]$. Subsequently, author in $[10]$ characterized completely idempotent and completely square-zero preserving maps. In this paper, we want to characterize the forms of
completely involution preserving maps between standard operator
algebras. Let $B(X)$ be the algebra of all
bounded linear operators on the Banach space $X$. Recall that a standard operator
algebra on $X$ is a norm closed subalgebra of $B(X)$ which
contains the identity and all finite-rank operators. Let $\mathcal{A} \subseteq B(X)$ be a
standard operator algebra. Denote $ \mathcal{P} =\{ A; A \in \mathcal{A}, A^2=A \}$, $ \Gamma =\{ A; A \in \mathcal{A}, A^2=I \}$. $A$ is involution whenever $A \in \Gamma $. Let $ \phi : \mathcal{M} \rightarrow \mathcal{N}$ be a map, where $ \mathcal{M}$ and $ \mathcal{N}$ are linear spaces. Define for each $n \in \mathbb{N} $, a map $ \phi _n : \mathcal{M} \otimes M_n \rightarrow \mathcal{N} \otimes M_n$ by
$$ \phi _n ((a_{ij})_{n \times n})=( \phi (a_{ij}))_{n \times n}.$$
Then $ \phi $ is said to be $n$-involution preserving if $ \phi _n $ preserves involution. $ \phi $ is said to be completely involution preserving if $ \phi _n $ preserves involution for each $n \in \mathbb{N} $. Our main
result is as follows.
%---------------------------------------------------------------------------------------%
\par \vspace{.3cm}
\textbf{Main Theorem } Let $X$ and $Y$ be infinite dimensional Banach spaces
and $\mathcal{A} \subseteq B(X)$, $\mathcal{B} \subseteq B(Y)$ be
standard operator algebras with $ \dim X \geq 3$. Suppose that $\phi:\mathcal{A}
\rightarrow \mathcal{B}$ is a surjective map. Then the following statements are equivalent: \\
$(1)$ $\phi$ is completely involution preserving in both directions. \\
$(2)$ $\phi$ is 2-involution preserving in both directions. \\
$(3)$ There exists a bijective bounded linear operator $A:X\longrightarrow Y$ such that
$$\phi(T)= \lambda _TATA^{-1} \hspace{.4cm} (T \in \mathcal{A}),$$
where $ \lambda _T=1$ or $-1$.
%---------------------------------------------------------------------------------------%

\section{Proofs}
In this section we prove our results. First we recall some
notations. We denote by $ \mathcal{P}_1( \mathcal{X})$ the set
of all rank-1 idempotent operators in $B(X)$. Let $X^ \prime$
denote the dual space of $X$. For every nonzero $x\in X$ and
$f\in X^ \prime$, the symbol $x\otimes f$ stands for the rank-1
linear operator on $X$ defined by $(x\otimes f)y=f(y)x$ for any
$y\in X$. Note that every rank-1 operator in $B(X)$ can be
written in this way. The rank-1 operator $x\otimes f$ is
idempotent if and only if $f(x)=1$. Given $P,Q\in \mathcal{P}$,  we say $P<Q$ if $PQ=QP=P$ and
$P\neq Q$. In addition, we say that $P$ and $Q$ are orthogonal if
$PQ=QP=0$. \par We need some
auxiliary lemmas to prove our main results. In order to prove main Theorem, it is enough only to prove $(2) \Rightarrow (3)$. So let $\mathcal{A} \subseteq B(X)$, $\mathcal{B} \subseteq B(Y)$ be standard operator algebras on infinite dimensional Banach spaces $X$ and $Y$, respectively. Suppose that $\phi:\mathcal{A}
\rightarrow \mathcal{B}$ is a surjective map such that $ \phi _2$ is an involution preserving map.
%---------------------------------------------------------------------------------------%
\begin{lemma} $\phi (I)=I$ or $\phi (I)=-I$ and $\phi (0)=0$.
\end{lemma}
%---------------------------------------------------------------------------------------%
\begin{proof} For any $T \in \mathcal{A}$,
$$ \begin{pmatrix}
     I & T \\
     0 & -I \\
   \end{pmatrix} \in \Gamma .$$
Thus $$ \phi _2(\begin{pmatrix}
     I & T \\
     0 & -I \\
   \end{pmatrix})= \begin{pmatrix}
     \phi (I) & \phi (T) \\
     \phi (0) & \phi (-I) \\
   \end{pmatrix} \in \Gamma $$
   So we obtain
   $$ \phi (I)^2+ \phi (T) \phi (0)=I \leqno(2.1)$$
   $$ \phi(I) \phi (T)+ \phi(T) \phi (-I)=0. \leqno(2.2)$$

Let $\phi (T_0)=I$. Taking $T=T_0$, Equation $(2.2)$ yields that $- \phi (I)= \phi (-I)$. This together with Equation $(2.2)$ yields that $ \phi(I) \phi (T)= \phi(T) \phi (I)$. Since $\phi$ is surjective, from previous equation can be concluded that $ \phi(I)$ commute with any element of $\mathcal{B}$ which implies that $ \phi (I)= \lambda I$ for a complex number $ \lambda$. Let $\phi (T_1)=0$. Taking $T=T_1$, Equation $(2.1)$ yields that $ \phi (I)^2= I$. Therefore $ \lambda ^2=1$ and so $\phi (I)=I$ or $\phi (I)=-I$. This together with Equation $(2.1)$ yields that $ \phi (T) \phi (0)=0$ for all $T \in \mathcal{A}$. Thus $ \phi (0)=0$.
\end{proof}
%---------------------------------------------------------------------------------------%
\begin{lemma} $\phi$ is injective.
\end{lemma}
%---------------------------------------------------------------------------------------%
\begin{proof} Let $T,S \in \mathcal{A}$ such that $ \phi (T)= \phi (S)$. So we have
\begin{eqnarray*}
 \begin{pmatrix}
     T & -I \\
     T^2-I & -T \\
   \end{pmatrix} \in \Gamma & \Rightarrow &
    \begin{pmatrix}
     \phi (T) & \phi (-I) \\
     \phi (T^2-I) & \phi (-T) \\
   \end{pmatrix} \in \Gamma \\ & \Rightarrow &
   \begin{pmatrix}
     \phi (S) & \phi (-I) \\
     \phi (T^2-I) & \phi (-T) \\
   \end{pmatrix} \in \Gamma \\ & \Rightarrow &
   \begin{pmatrix}
     S & -I \\
     T^2-I & -T \\
   \end{pmatrix} \in \Gamma
\end{eqnarray*}
which implies that $T=S$ and this completes the proof.
\end{proof}
%---------------------------------------------------------------------------------------%
\begin{lemma} For any $T,S \in \mathcal{A}$ we have $ \phi (T) \phi (S)= \phi (S) \phi (T)=0$ if $TS=ST=0$.
\end{lemma}
%---------------------------------------------------------------------------------------%
\begin{proof} We have $TS=ST=0$ if and only if
  $ \begin{pmatrix}
     I & T \\
     S & -I \\
   \end{pmatrix} \in \Gamma $.
   This together with Lemma 2.1 and the preserving property of $ \phi _2$ yields
   $$TS=ST=0 \Rightarrow \phi (T) \phi (S)= \phi (S) \phi (T)=0.$$
\end{proof}
%---------------------------------------------------------------------------------------%
\par Next assume that $\phi (I)=I$. We may replace $\phi$ by $ -\phi $ if $\phi (I)=-I$
%---------------------------------------------------------------------------------------%
\begin{lemma} $\phi$ preserves the idempotent operators.
\end{lemma}
%---------------------------------------------------------------------------------------%
\begin{proof} For any $T \in \mathcal{A}$ we have
$$
 \begin{pmatrix}
     T & I \\
     I-T^2 & -T \\
   \end{pmatrix} \in \Gamma  \Rightarrow
    \begin{pmatrix}
     \phi (T) & I \\
     \phi (I-T^2) & \phi (-T) \\
    \end{pmatrix} \in \Gamma $$
    which implies that
    $$ \phi (T)^2+  \phi (I-T^2)=I. \leqno(2.3)$$
    If $T$ is an idempotent, then from (2.3) we obtain
    $$ \phi (T)^2+  \phi (I-T)=I. \leqno(2.4)$$
   If $T$ is an idempotent, then $T(I-T)=0$ which by Lemma 2.3 implies
    $$ \phi (T) \phi (I-T)=0. \leqno(2.5)$$
Multiplying (2.4) from right by $ \phi (I-T)$ and then changing $I-T$ to $T$, we see that
$$ \phi (T)^2= \phi (T).$$
 Therefore the proof is complete.
\end{proof}
%---------------------------------------------------------------------------------------%
\begin{lemma} $\phi$ preserves the orthogonality and the order of idempotent operators.
\end{lemma}
%---------------------------------------------------------------------------------------%
\begin{proof} From Lemma 2.3 and Lemma 2.4 we can conclude that $\phi$ preserves the orthogonality of idempotent operators.
 \par
 By Lemma 2.4, Lemma 2.1 and Equation (2.3) we obtain
   $$ \phi(I-P)=I- \phi (P) \hspace{.4cm} (P \in \mathcal{P}). \leqno(2.6)$$
   Note that $P<Q$ if and only if $P \perp I-Q$. So by (2.6) and previous part we obtain $P<Q \Rightarrow \phi (P)< \phi (Q)$.
  \end{proof}
%---------------------------------------------------------------------------------------%
   \begin{remark} If we assume $ \phi _2$ is an involution preserving map in both directions, then applying a similar argument to $ \phi ^{-1}$ yields converse the same results
   (the above Theorem) because $ \phi$ is a injective map and $ \phi ^{-1}$ has the same property of $\phi$. Hence $\phi$ preserves the orthogonality and the order of idempotent
   operators in both directions.
\end{remark}
%---------------------------------------------------------------------------------------%
\par
  Next assume that $ \phi _2$ is an involution preserving map in both directions.
%---------------------------------------------------------------------------------------%
\begin{lemma} For any $T\in \mathcal{A}$ there exists a complex number $ \lambda _T$ such that $ \phi (T)= \lambda _TATA^{-1}$, where $A:X\longrightarrow Y$ is a bijective bounded linear operator.
\end{lemma}
%---------------------------------------------------------------------------------------%
\begin{proof}
Lemmas 2.2 and 2.5 imply that $\phi$ is a bijection preserving the orthogonality of idempotents in both directions. It follows from lemma 3.1 in $[10]$ that there exists a bijective bounded linear or (in the complex case) conjugate linear operator $A:X\longrightarrow Y$ such that
$$\phi(P)=APA^{-1} \hspace{.4cm} (P \in \mathcal{P}_{ \mathcal{A}})$$
or a bijective bounded linear or (in the complex case) conjugate linear operator $A:X ^{ \prime}\longrightarrow Y$ such that
$$\phi(P)=AP^*A^{-1} \hspace{.4cm} (P \in \mathcal{P}_{ \mathcal{A}}).$$
We show that the second case can not occur. Assume on the contrary that $\phi(P)=AP^*A^{-1}$ for all $P \in \mathcal{P}_{ \mathcal{A}}$. Let $x,y \in X$ be two arbitrary vector such that are linearly independent. So there exist $f_1,f_2 \in X ^{ \prime}$ such that $f_1(x)=f_2(y)=1$ and $f_1(y)=f_2(x)=0$. If $f=f_1-f_2$ and $g=f_1+f_2$, then we have $f(x)=-f(y)=g(x)=g(y)=1$. Hence we have
$$
    \begin{pmatrix}
     I-x \otimes f & x \otimes g \\
      -y \otimes f & I-y \otimes g \\
    \end{pmatrix} \in \Gamma $$
    which implies that
    $$
    \begin{pmatrix}
     A & 0 \\
     0 & A \\
    \end{pmatrix}
    \begin{pmatrix}
     I-f \otimes x & g \otimes x \\
      -f \otimes y & I-g \otimes y \\
    \end{pmatrix}
    \begin{pmatrix}
     A^{-1} & 0 \\
     0 & A^{-1} \\
    \end{pmatrix} \in \Gamma .$$
    It is clear that this is a contradiction. Therefore, we have $\phi(P)=APA^{-1}$ for all $P \in \mathcal{P}_{ \mathcal{A}}$. It is trivial that without loss of generality, we can suppose that $\phi(P)=P$ for all $P \in \mathcal{P}_{ \mathcal{A}}$. \par For any $T \in \mathcal{A}$ we have
$$
    \begin{pmatrix}
     \phi (T) & I \\
     \phi (I-T^2) & \phi (-T) \\
    \end{pmatrix} \in \Gamma $$
    which implies that
    $$ \phi (-T)=- \phi (T). \leqno(2.7)$$
On the other hand, for any $T,S \in \mathcal{A}$ we have
$$
 \begin{pmatrix}
     I-TS & -T \\
     STS-2S & ST-I \\
   \end{pmatrix} \in \Gamma .$$
   Hence
   $$
 \begin{pmatrix}
     \phi (I-TS) & \phi (-T) \\
     \phi (STS-2S) & \phi (ST-I) \\
   \end{pmatrix} \in \Gamma .$$
   So we obtain
   $$ \phi (I-TS) \phi (-T)+ \phi (-T) \phi (ST-I)=0.$$
   Let $T$ be an operator and $x$ be an arbitrary nonzero vector of $X$ such that $x \not \in \ker T$. So there exists a nonzero functional $f$ such that $f(Tx)=1$. Let $S=x \otimes f$ in previous equation and then using (2.6) and (2.7) yields
   $$Tx \otimes f \phi (T)= \phi (T)x \otimes fT$$
   which implies that $ \phi (T)x$ and $Tx$ are linearly dependent for all $x \in X$ such that $x \not \in \ker T$. So it is clear that $ \phi (T)x$ and $Tx$ are linearly dependent for all $x \in X$. Thus we can conclude from $[2]$ that there exists a complex number $ \lambda _T$ such that
   $$ \phi (T)= \lambda _TT \leqno(2.8)$$ for all operator $T$ such that $T$ is not rank one or there exist $x \in X$ and $f,g \in X ^{ \prime}$ such that $T=x \otimes f$ and $ \phi (T)=x \otimes g$ . \par Let $P$ be a rank one idempotent operator. We have
   $$
 \begin{pmatrix}
     I-2P & P \\
     0 & I \\
   \end{pmatrix} \in \Gamma $$
   which by (2.8) and Lemma 2.1 yields
   $$
 \begin{pmatrix}
     \lambda _{I-2P}(I-2P) & P \\
     0 & I \\
   \end{pmatrix} \in \Gamma .$$
   Therefore we obtain $ \lambda _{I-2P}=1$ and so
   $$ \phi(I-2P)=I-2P \leqno(2.9)$$
   for all rank one idempotent $P$. \par
   By previous descriptions, for any rank one operator $x \otimes f$, we can write $ \phi (x \otimes 2f)= x \otimes h$. We can find $y \in X$ such that $f(y)=1$. Then
   $$
 \begin{pmatrix}
     I &  -x \otimes 2f\\
     0 & I-2y \otimes f \\
   \end{pmatrix} \in \Gamma $$
   which by (2.7) and (2.9) and Lemma 2.1 yields
   $$
 \begin{pmatrix}
     I & -x \otimes h \\
     0 & I-2y \otimes f \\
   \end{pmatrix} \in \Gamma .$$
   Hence we obtain
   $$x \otimes h=h(y)y \otimes f$$
   which implies that $f$ and $h$ are linearly dependent. So for any rank one operator $x \otimes f$ there exists a complex number $ \lambda _{x \otimes f}$ such that
   $$ \phi (x \otimes f)= \lambda _{x \otimes f}x \otimes f.$$
   This together with (2.8) completes the proof.
\end{proof}
%---------------------------------------------------------------------------------------%
\par \vspace{.3cm}
\textbf{Proof of Main Theorem }
 As stated in the proof of Lemma 2.7, without loss of generality, we can suppose that $ \phi (T)= \lambda _TT$ for a complex number $ \lambda _T$. So it is enough to prove that $ \lambda _T=1$ or $  \lambda _T=-1$ for any $T \in \mathcal{A}$. By Equation (2.3) we have
$$ \lambda _{I-T^2}(I-T^2)=I-{ \lambda _T}^2T^2.$$
If $I$ and $T^2$ are linearly independent, then $ \lambda _{I-T^2}={ \lambda _T}^2=1$ which implies that $ \lambda _T=1$ or $  \lambda _T=-1$. Let $ \lambda _T=1$. Let $ \mu $ be a nonzero complex number. For any $S \in \mathcal{A}$ we have
$$
 \begin{pmatrix}
     I- \mu S & \mu I \\
     2S- \mu S^2 & \mu S-I \\
   \end{pmatrix} \in \Gamma .$$
   Hence
   $$
 \begin{pmatrix}
     \phi (I- \mu S) & \phi ( \mu I) \\
     \phi (2S- \mu S^2) & \phi ( \mu S-I) \\
   \end{pmatrix} \in \Gamma .$$
   So we obtain
   $$ \phi (I- \mu S)^2+ \phi ( \mu I) \phi (2S- \mu S^2)=I.$$
   If $I$, $S$, $S^2$, $S^3$ and $S^4$ are linearly independent, then the sets $ \{I, (I- \mu S)^2 \}$ and $ \{I, (2S- \mu S^2)^2 \}$ contain the linearly independent vectors. So by the previous part we obtain
   $$(I- \mu S)^2+ \lambda _{ \mu I}  \mu I(2S- \mu S^2)=I,$$
   which yields
   $$( \mu ^2- \lambda _{ \mu I}  \mu ^2)S^2+(-2 \mu +2 \mu \lambda _{ \mu I})S=0.$$
   Since $S$ and $S^2$ are linearly independent, we obtain $ \lambda _{ \mu I}=1$. \par
   Now let $T^2= \alpha I$ for a nonzero complex number $ \alpha $. By Equation (2.3) we have
   $$(1- \alpha ) \lambda _{(1- \alpha )I}I= I-{ \lambda _T}^2 \alpha I$$
   which by previous part we obtain $1- \alpha=1-{ \lambda _T}^2 \alpha $. Thus $ \lambda _T=1$ or $  \lambda _T=-1$. For another case $  \lambda _T=-1$, we can use the similar discussion. The proof is complete.
%---------------------------------------------------------------------------------------%
\par \vspace{.4cm}{\bf Acknowledgements:} This research is partially
supported by the Research Center in Algebraic Hyperstructures and
Fuzzy Mathematics, University of Mazandaran, Babolsar, Iran.

%---------------------------------------------------------------------------------------%
\bibliographystyle{amsplain}

\end{document}